\documentclass[12pt]{amsart}
\usepackage{amssymb}

\input xy
\xyoption{all}

\newtheorem{theorem}{Theorem}[section]
\newtheorem{lemma}[theorem]{Lemma}
\newtheorem{proposition}[theorem]{Proposition}

\numberwithin{equation}{section}

\newcommand{\et}{{\acute{e}tale}}

\begin{document}

\baselineskip=15pt

\title[Brauer group and moduli of parabolic bundles]{Brauer
group of a moduli space of parabolic vector bundles over a curve}

\author[I. Biswas]{Indranil Biswas}

\address{School of Mathematics, Tata Institute of Fundamental
Research, Homi Bhabha Road, Bombay 400005, India}

\email{indranil@math.tifr.res.in}

\author[A. Dey]{Arijit Dey}

\address{School of Mathematics, Tata Institute of Fundamental
Research, Homi Bhabha Road, Bombay 400005, India}

\email{arijit@math.tifr.res.in}



\date{}

\begin{abstract}
Let ${\mathcal P}{\mathcal M}^\alpha_s$ be a moduli space of
stable parabolic vector bundles of rank $n\, \geq\,2$ and fixed 
determinant of degree $d$
over a compact connected Riemann surface $X$ of genus
$g(X)\, \geq\, 2$. If $g(X)\, =\,2$, then we assume that
$n\, >\, 2$. Let $m$ denote the greatest common divisor
of $d$, $n$ and the dimensions of all the successive quotients
of the quasi--parabolic filtrations. We prove that the
Brauer group ${\rm Br}({\mathcal P}{\mathcal
M}^\alpha_s)$ is isomorphic to the cyclic group ${\mathbb Z}/
m{\mathbb Z}$. We also show that ${\rm Br}({\mathcal 
P}{\mathcal M}^\alpha_s)$ is generated by the Brauer class of
the Brauer--Severi variety over ${\mathcal P}{\mathcal
M}^\alpha_s$ obtained by restricting the universal
projective bundle over $X\times {\mathcal P}{\mathcal M}^\alpha_s$.
\end{abstract}

\maketitle

\section{Introduction} \label{Intro}

Let $Y$ be a smooth quasi--projective variety over $\mathbb C$.
The cohomological Brauer group of $Y$ is defined to be
$H^2(Y_{\et},\,{\mathbb G}_m)_{\rm torsion}$, and it is denoted by
$\text{Br}'(Y)$. It is known that the group
$H^2(Y_{\et},\,{\mathbb G}_m)$ is torsion. The Brauer group of
$Y$, which is denoted by $\text{Br}(Y)$, is defined to be the Morita 
equivalence classes of Azumaya algebras over $Y$. Giving an Azumaya 
algebra
over $Y$ is equivalent to giving a Brauer--Severi variety over $Y$ which 
is also equivalent to giving a principal
$\text{PGL}_{\mathbb C}$--bundle over $Y$. Given a
principal $\text{PGL}_{\mathbb C}(r)$--bundle $E_{\text{PGL}_{\mathbb 
C}(r)}\, \longrightarrow\, Y$, using the cohomology exact sequence 
for the short exact sequence
$$
e\, \longrightarrow\, {\mu}_r\, \longrightarrow\, \text{SL}_{\mathbb 
C}(r) \, \longrightarrow\,\text{PGL}_{\mathbb C}(r) \, \longrightarrow
\, e
$$
we get an element of $\text{Br}'(Y)$. The resulting homomorphism
$\text{Br}(Y)\, \longrightarrow\, \text{Br}'(Y)$ is injective. A
theorem due to Gabber says that this homomorphism is surjective.
Therefore, the three groups, namely $\text{Br}(Y)$, 
$\text{Br}'(Y)$, and $H^2(Y_{\et},\,{\mathbb G}_m)$, coincide.

A Brauer--Severi variety over $Y$
is the projectivization of a vector bundle over $Y$ if and only
if its class in $\text{Br}(Y)$ vanishes. If $U$ a Zariski open subset of
$Y$ such that the codimension
of the complement $Y\setminus U$ is at least two, then the natural
restriction
homomorphism $\text{Br}(Y)\,\longrightarrow\, \text{Br}(U)$ is an
isomorphism. (See \cite{Mil}, \cite{Gr1}, \cite{Gr2}, \cite{Gr3} for
the above mentioned properties and more.)

Let $X$ be an irreducible smooth complex projective curve
of genus $g(X)$, with $g(X)\, \geq\,2$. Fix an
integer $n$, with $n\, \geq\, 2$. If $g(X)\, =\,2$, then
we assume that $n\, \geq\, 3$. Fix distinct points
\begin{equation}\label{ppoint}
\{p_1\, ,\cdots\, , p_\ell\}\, \subset\, X
\end{equation}
with $\ell\, \geq\,1$.
For each $i\, \in\, [1\, ,\ell]$, fix
positive integers $\{r_{i,1}\, , \cdots\, , r_{i,a_i}\}$
such that
\begin{equation}\label{n}
\sum_{j=1}^{a_i} r_{i,j} \, =\, n\, .
\end{equation}
To each pair $(i\, ,j)$, $i\, \in\, [1\, ,\ell]$ and
$j\, \in\, [1\, ,a_i]$, associate a real number
$\alpha_{i,j}$ satisfying the following two conditions:
\begin{itemize}
\item $0\, \leq\, \alpha_{i,j}\, <\, 1$, and

\item $\alpha_{i,j}\, >\, \alpha_{i,j'}$ whenever $j\, >\, j'$.
\end{itemize}
We also fix a line bundle $L$ over $X$. The degree of $L$
will be denoted by $d$.

We consider parabolic vector bundles over $X$
of the following type.

Let $E$ be a vector bundle of rank $n$ with
$\det(E)\, :=\, \bigwedge^n E\, =\, L$. The parabolic points
are $\{p_1\, ,\cdots\, , p_\ell\}$ (see \eqref{ppoint}).
The \textit{quasi--parabolic filtration} of $E_{p_i}$
is of the form 
\begin{equation}\label{fi.}
E_{p_i} \,=:\, F_{i,1} \, \supsetneq\, \cdots\,
\, \supsetneq\, F_{i,a_i}\, \not=\, 0
\end{equation}
with $\dim F_{i,j}\, =\, \sum_{k=j}^{a_i} r_{i,k}$.
The parabolic weight of $F_{i,j}$ is $\alpha_{i,j}$.
(See \cite{MS}, \cite{MY} for more details on
parabolic bundles.)

Let ${\mathcal P}{\mathcal M}^\alpha_s$ denote the moduli
space of stable parabolic vector bundles of the above
type. This moduli space ${\mathcal P}{\mathcal M}^\alpha_s$
is a smooth quasi--projective variety. (See \cite{MS},
\cite{MY} for the construction of
${\mathcal P}{\mathcal M}^\alpha_s$.)

Over $X\times {\mathcal P}{\mathcal M}^\alpha_s$, there is
a unique universal projective bundle, which we will
denote by $\mathbb P$. For any $E_*\, \in\, {\mathcal P}{\mathcal 
M}^\alpha_s$, the restriction of $\mathbb P$ to
$X\times \{E_*\}$ is identified with the projectivization
$P(E)$, where $E$ is the underlying vector bundle
for the parabolic vector bundle $E_*$. Fix a point $x_0\, \in\,
X$. Let
\begin{equation}\label{pr}
{\mathbb P}_{x_0}\, :=\, {\mathbb P}\vert_{\{x_0\}\times
{\mathcal P}{\mathcal M}^\alpha_s}\, \longrightarrow\,
{\mathcal P}{\mathcal M}^\alpha_s
\end{equation}
be the projective bundle over ${\mathcal P}{\mathcal M}^\alpha_s$.

We prove the following theorem:

\begin{theorem}\label{thm0}
The Brauer group ${\rm Br}({\mathcal P}{\mathcal
M}^\alpha_s)$ is isomorphic to the cyclic group
${\mathbb Z}/m{\mathbb Z}$, where
$$
m\, =\, {\rm g.c.d.}(d\, , n\, ,r_{1,1}\, , \cdots\, , 
r_{1,a_1}\, , \cdots\, , r_{\ell,1}\, , \cdots\, ,
r_{\ell,a_\ell})
$$
(as before, $d\,=\, {\rm degee}(L)$).

The cyclic group ${\rm Br}({\mathcal P}{\mathcal
M}^\alpha_s)$ is generated by the Brauer class of the
Brauer--Severi variety ${\mathbb P}_{x_0}$ in \eqref{pr}.
\end{theorem}

The Brauer group of a smooth complex quasi--projective variety $Y$ 
parametrizes the
equivalence classes of principal $\text{PGL}_r({\mathbb
C})$--bundles, $r\, \geq\, 1$,
over $Y$. We recall that a principal
$\text{PGL}_r({\mathbb C})$--bundle $P$ is equivalent
to a principal $\text{PGL}_{r'}({\mathbb C})$--bundle $P'$
if there are vector bundles $F$ and $F'$ of ranks
$r$ and $r'$ respectively, such that the two
principal $\text{PGL}_{rr'}({\mathbb C})$--bundles
$P\bigotimes P(F')$ and
$P(F)\bigotimes P'$ are isomorphic. As mentioned in the introduction,
the Brauer group of $Y$ coincides with the cohomological
Brauer group $\text{Br}'(Y)$.

\section{Variation of moduli space of parabolic bundles}\label{tha}

In this section we recall a result of Thaddeus which will
be crucially used here.

As before, $X$ is an irreducible smooth complex projective curve
of genus $g(X)$, with $g(X)\, \geq\,2$.
We will compare two different moduli spaces of stable
parabolic vector bundles on $X$. For that,
fix a rank, a determinant line bundle (top exterior product),
a nonempty set of parabolic points of $X$ and
quasi--parabolic filtration types over the parabolic points;
if $g(X)\, =\,2$, then fix the rank to be at least three.
Take two different sets of parabolic weights, say $\Lambda_1$
and $\Lambda_2$. Let ${\mathcal P}{\mathcal M}^{\Lambda_1}$ 
and ${\mathcal P}{\mathcal M}^{\Lambda_2}$ be the corresponding
moduli spaces of stable parabolic vector bundles.

\begin{lemma}\label{thaddeus}
There is a smooth quasi--projective complex variety $\mathcal U$
and open embeddings
$$
\varphi_i\, :\, {\mathcal U}\,\hookrightarrow\,
{\mathcal P}{\mathcal M}^{\Lambda_i}\, ,
$$
$i\,=\, 1\, ,2$, such that the codimension of the
complement
$$
{\mathcal P}{\mathcal M}^{\Lambda_i}\setminus
\varphi_i({\mathcal U})\, \subset\,
{\mathcal P}{\mathcal M}^{\Lambda_i}
$$
is at least two.
\end{lemma}

\begin{proof}
For $i\,=\, 1\, ,2$, let
$\overline{{\mathcal P}{\mathcal M}}^{\Lambda_i}$ be the
moduli space of semistable parabolic vector bundles of the
above type with
parabolic weights $\Lambda_i$. The moduli space
$\overline{{\mathcal P}{\mathcal M}}^{\Lambda_i}$ is
normal, and ${\mathcal P}{\mathcal M}^{\Lambda_i}$ is the
smooth locus of $\overline{{\mathcal P}{\mathcal M}}^{\Lambda_i}$.
Thaddeus proved that there are subschemes
${\mathcal S}_1\, \subset\, \overline{{\mathcal P}{\mathcal 
M}}^{\Lambda_1}$ and ${\mathcal S}_2\, \subset\, \overline{{\mathcal 
P}{\mathcal M}}^{\Lambda_2}$, of codimension at--least two, such
that the blow--up of $\overline{{\mathcal P}{\mathcal
M}}^{\Lambda_1}$ along ${\mathcal S}_1$ is isomorphic to the
blow--up of $\overline{{\mathcal P}{\mathcal
M}}^{\Lambda_2}$ along ${\mathcal S}_2$ \cite[Section 7]{Th1}
(see also \cite{Th2}). The lemma follows immediately from
this result.
\end{proof}

\section{One parabolic point and small parabolic weights} 
\label{smallweight}

In this section we assume the following:
\begin{itemize}
\item there is only one parabolic point, so $\ell\, =\, 1$, and

\item all the parabolic weights are sufficiently small
(smaller than $1/n^2$).
\end{itemize}
If $a_1\, =\, 1$ (see \eqref{n}), then ${\mathcal P}{\mathcal 
M}^\alpha_s$ coincides with the moduli space of stable vector
bundles of rank $n$ and determinant $L$, and the
Brauer group of it is already computed in
\cite{BBGN}. Hence we will assume that $a_1\, >\, 1$.

Let $\overline{\mathcal N}$ denote the moduli space of semistable
vector bundles $E$ of rank $n$ over $X$ with $\det(E)\, :=\,
\bigwedge^n E\, =\, L$. Let
\begin{equation}\label{s}
{\mathcal N}\, \subset\, \overline{\mathcal N}
\end{equation}
be the nonempty
Zariski open subset that parametrizes the stable ones.

Since the parabolic weights are sufficiently small we know the 
following:
\begin{itemize}
\item For a stable parabolic bundle, the underlying vector bundle 
is semistable.

\item For any quasi--parabolic structure on a stable vector
bundle of rank $n$ and determinant $L$, the corresponding
parabolic vector bundle is stable.
\end{itemize}
Therefore, we have a forgetful morphism
\begin{equation}\label{pi}
\pi_0\, :\, {\mathcal P}{\mathcal M}^\alpha_s \,\longrightarrow\,
\overline{\mathcal N}
\end{equation}
that sends any stable parabolic bundle to its underlying vector
bundle. Let
\begin{equation}\label{p2}
M\, :=\, \pi^{-1}_0({\mathcal N})\, \subset\,
{\mathcal P}{\mathcal M}^\alpha_s
\end{equation}
be the inverse image, where ${\mathcal N}$ is defined in
\eqref{s}. Note that $M$ is a Zariski open dense subset of
${\mathcal P}{\mathcal M}^\alpha_s$. Let
\begin{equation}\label{p3}
\pi\, :=\, \pi_0\vert_M \, :\, M\,\longrightarrow\,
{\mathcal N}
\end{equation}
be the restriction of $\pi_0$ constructed in \eqref{pi}.

Let
\begin{equation}\label{P}
P\, \subset\, \text{SL}(n, {\mathbb C})
\end{equation}
be the parabolic subgroup
that preserves a fixed filtration of subspaces
\begin{equation}\label{P0}
{\mathbb C}^n \,= \, C_{1} \, \supsetneq\, \cdots\,
\, \supsetneq\, C_{a_1}
\end{equation}
such that $\dim C_j\, =\, \sum_{k=j}^{a_1} r_{1,k}$
(see \eqref{fi.}).
We noted earlier that
any quasi--parabolic structure on a stable vector
bundle of rank $n$ and determinant $L$ lies
in ${\mathcal P}{\mathcal M}^\alpha_s$.
Therefore, the projection $\pi$ in \eqref{p3}
defines a fiber bundle over ${\mathcal N}$ with fiber
$\text{SL}(n, {\mathbb C})/P$, where $P$ is the subgroup
in \eqref{P}.

Consider the moduli space ${\mathcal N}$ in \eqref{s}.
There is a universal projective bundle ${\mathbb P}^0$
on $X\times {\mathcal N}$ \cite[p. 6,
Theorem 2.7]{BBNN}. For the very general point
$E\, \in\, {\mathcal N}$, the projective bundle
$P(E)$ does not admit any nontrivial
automorphism. Hence ${\mathbb P}^0$ is unique.
Fix a point $x_0\, \in\, X$. Let
\begin{equation}\label{i1.}
{\mathbb P}^0_{x_0}\, \longrightarrow\, {\mathcal N}
\end{equation}
be the projective bundle obtained by restricting
${\mathbb P}^0$ to $\{x_0\}\times {\mathcal N}$.
Let
\begin{equation}\label{is.}
{\mathbb P}'\, :=\, \pi^*{\mathbb P}^0_{x_0}
\end{equation}
be the projective bundle over $M$, where $\pi$
is the projection in \eqref{p3}.

Define
\begin{equation}\label{m}
m\, :=\, {\rm g.c.d.}(d\, , n\, , r_{1,1}\, , \cdots\, ,
r_{1,a_1})\, ,
\end{equation}
where $r_{1,j}$ are as in \eqref{n} and
$d\, =\, \text{degree}(L)$.

\begin{lemma}\label{lemma1}
The Brauer group ${\rm Br}(M)$ of the smooth quasi--projective
variety $M$ in \eqref{p2} is the cyclic group 
${\mathbb Z}/m{\mathbb Z}$,
where $m$ is defined in \eqref{m}. The group
${\rm Br}(M)$ is generated by the 
the Brauer class of the Brauer--Severi variety ${\mathbb P}'$
constructed in \eqref{is.}.
\end{lemma}

\begin{proof}
Applying the Leray spectral sequence (see \cite[p. 89, Theorem 
1.18]{Mil}) to the projection $\pi$ in \eqref{p3},
\begin{equation}\label{eq0}
E^{p,q}_2\,:=\, H^p({\mathcal N}_\et,\, R^q\pi_*\mathbb G_m)
\implies H^{p+q}(M_\et,\, \mathbb G_m)\, ,
\end{equation}
where $\mathbb G_m$ denotes the sheaf of regular invertible
function. Consequently, we have a long exact sequence 
\begin{equation}\label{eqn1}
\longrightarrow \,E^{0,1}_2 \,\stackrel{\theta}{\longrightarrow}\,
E^{2,0}_2\,\stackrel{\theta_1}{\longrightarrow}\, H^{2}(M_\et,\,
\mathbb G_m)\,\longrightarrow\, E^{0,2}_2 \, \longrightarrow\, .
\end{equation}

We will investigate the first term in \eqref{eqn1}.

The variety $\mathcal N$ is simply connected \cite[p. 266,
Proposition 1.2(b)]{BBGN}, and
$$
\text{Pic}(\text{SL}(n, {\mathbb C})/P)\, =\,
{\mathbb Z}^{\oplus (a_1-1)}
$$
(see \eqref{P}),
where $a_1$ is the integer in \eqref{P0}. Hence
$$
R^1\pi_*\mathbb G_m\, \longrightarrow\, {\mathcal N}
$$
is the constant sheaf with stalk isomorphic to
${\mathbb Z}^{\oplus (a_1-1)}$. From the connectedness of 
$\text{SL}(n, {\mathbb C})$ it follows that the automorphisms
of $\text{SL}(n, {\mathbb C})/P$ given by the left
translation action
of $\text{SL}(n, {\mathbb C})$ on it act trivially on
$\text{Pic}(\text{SL}(n, {\mathbb C})/P)$. Consequently,
the direct image
$$
R^1\pi_*\mathbb G_m\, \longrightarrow\, {\mathcal N}
$$
is canonically identified with the constant sheaf with
stalk $\text{Pic}(\text{SL}(n, {\mathbb C})/P)$.

The homogeneous variety $\text{SL}(n, {\mathbb C})/P$
(see \eqref{P}) is smooth complete and rational \cite{BR}. Hence
\begin{equation}\label{gp}
H^2((\text{SL}(n, {\mathbb C})/P)_\et,\, \mathbb G_m)\, =\, 0\, .
\end{equation}
We will show that
\begin{equation}\label{e1}
(R^2\pi_*\mathbb G_m)_{\rm torsion}\,=\, 0\, .
\end{equation}

Take any integer $\delta\, \geq\, 2$. From the Kummer sequence,
\begin{equation}\label{e2}
R^1\pi_*\mathbb G_m\, \longrightarrow\, R^2\pi_* \mu_\delta
\, \longrightarrow\, (R^2\pi_*\mathbb G_m)[\delta]\, \longrightarrow\, 
0\, ,
\end{equation}
where $(R^2\pi_*\mathbb G_m)[\delta]\, \subset\, R^2\pi_*\mathbb G_m$
is the subgroup generated by $\delta$--torsion elements. Take
any geometric point $z\, \longrightarrow\, {\mathcal N}$,
and fix a Henselization of ${\mathcal N}$ at $z$. Now pull
back \eqref{e2} to $z$. From the above description of
$R^1\pi_*\mathbb G_m$ as the constant sheaf with stalk isomorphic
to ${\mathbb Z}^{\oplus (a_1-1)}$, we know that
$(R^1\pi_*\mathbb G_m)_z \,=\, \text{Pic}(\pi^{-1}(z))$. From the
proper base change theorem \cite[p. 223, Theorem 2.1]{Mil},
$$
(R^2\pi_* \mu_\delta)_z\, =\, H^2(\pi^{-1}(z),\, \mu_\delta)\, .
$$
The homomorphism $\text{Pic}(\pi^{-1}(z))\, \longrightarrow\,
H^2(\pi^{-1}(z),\, \mu_\delta)$ is surjective due to \eqref{gp}
in the long exact sequence of cohomologies for the Kummer sequence.
Hence $((R^2\pi_*\mathbb G_m)[\delta])_z\,=\, 0$ for all
$z$. This implies that \eqref{e1} holds.

Note that $R^0\pi_*{\mathbb G}_m\, =\, {\mathbb G}_m$.
Hence $E^{2,0}_2\, =\, H^{2}({\mathcal N}_\et,\,
\mathbb G_m)$ (see \eqref{eq0}).
The cohomological Brauer group $H^{2}({\mathcal N}_\et,\,
\mathbb G_m)_{\rm torsion}$ is generated by the Brauer class of the
projective bundle ${\mathbb P}^0_{x_0}$ in \eqref{i1.}
\cite[p. 266, Proposition 1.2(a)]{BBGN}. From the
construction of ${\mathbb P}'$ in \eqref{is.} it follows
that the Brauer class of ${\mathbb P}'$ coincides with
$\theta_1(\beta)$, where $\theta_1$ is the homomorphism in
\eqref{eqn1}, and $\beta\, \in\, H^{2}({\mathcal N}_\et,\,
\mathbb G_m)$ is the Brauer class of ${\mathbb P}^0_{x_0}$.
Therefore, from \eqref{eqn1} and
\eqref{e1} we conclude that the Brauer class of
${\mathbb P}'$ generates $H^{2}(M_\et,\, \mathbb G_m)_{\rm torsion}$.

Therefore, from \eqref{eqn1} and \eqref{e1},
\begin{equation}\label{eqn2}
\longrightarrow \,\text{Pic}(\text{SL}(n, {\mathbb C})/P) 
\,\stackrel{\theta}{\longrightarrow}\, 
\text{Br}({\mathcal N})\,\stackrel{\theta_1}{\longrightarrow}
\,\text{Br}(M) \,\longrightarrow\, 0\, .
\end{equation}

We will describe generators of $\text{Pic}(\text{SL}(n, {\mathbb 
C})/P)$. For each $j\, \in\, [2\, , a_1]$, define
\begin{equation}\label{cj}
c_j\, :=\, \sum_{k=j}^{a_1} r_{1,k}
\end{equation}
(see \eqref{n}). Let
\begin{equation}\label{fj}
f_j\, :\, \text{SL}(n, {\mathbb C})/P\,
\longrightarrow\, P(\bigwedge\nolimits^{c_j} {\mathbb C}^n)
\end{equation}
be the morphism that sends any filtration
$$
{\mathbb C}^n \,= \, V_{1} \, \supsetneq\, \cdots\,
\, \supsetneq\, V_{a_1}
$$
to the line in $\bigwedge^{c_j} {\mathbb C}^n$ defined
by $\bigwedge^{c_j} V_j$. Let
\begin{equation}\label{f2}
\zeta_j\, :=\, f^*_j{\mathcal O}_{P(\wedge^{c_j} {\mathbb C}^n)}
(1)\, \longrightarrow\, \text{SL}(n, {\mathbb C})/P
\end{equation}
be the line bundle, where $f_j$ is the morphism in \eqref{fj}.
It is known that the homomorphism
\begin{equation}\label{f3}
\eta\, :\, {\mathbb Z}^{\oplus (a_1-1)}\, \longrightarrow\,
\text{Pic}(\text{SL}(n, {\mathbb C})/P)
\end{equation}
defined by
$$
(z_1\, , \cdots\, ,z_{a_1-1})\, \longmapsto\,
\bigotimes_{j=1}^{a_1-1}\zeta^{\otimes z_j}_{j+1}
$$
is an isomorphism.

Let $\beta\, \in\, H^{2}({\mathcal N}_\et,\, \mathbb G_m)$ be
the Brauer class of the projective bundle ${\mathbb P}^0_{x_0}$
defined in \eqref{i1.}. The order of the generator
$\beta$ of $H^{2}({\mathcal N}_\et,\, \mathbb G_m)$
is $\text{g.c.d.}(n,d)$ \cite[p. 267, Theorem 1.8]{BBGN}.
Consider the homomorphism $\theta$ in \eqref{eqn2}.
For each $j\,\in [2\, ,a_1]$, we have
\begin{equation}\label{c.}
\theta(\zeta_j) \, =\, c_j\cdot\beta\, ,
\end{equation}
where $c_j$ is defined in \eqref{cj} \cite[p. 203, Proposition
4.4(ii)]{Art} (see also \cite[p. 267, Lemma 1.5]{BBGN}). From
\eqref{c.} it follows that the order of the generator
$\theta_1(\beta)\, \in \,\text{Br}(M)$ is
\begin{equation}\label{dm}
m'\, :=\, \text{g.c.d.}(\text{g.c.d.}(n,d)\, ,c_2\, , \cdots \, ,
c_{a_1})\, ,
\end{equation}
where $\theta_1$ is the homomorphism in \eqref{eqn2}.

Using \eqref{n} it follows that $m\, =\, m'$,
where $m$ is defined in \eqref{m}, and $m'$ is defined in
\eqref{dm}. Therefore, the order of the
generator $\theta_1(\beta)\, \in \,\text{Br}(M)$ is $m$.
We noted earlier that the Brauer class of ${\mathbb P}'$
(defined in \eqref{is.}) coincides with $\theta_1(\beta)$.
This completes the proof of the lemma.
\end{proof}

\begin{lemma}\label{lem2a}
Consider the Zariski open subset $M\, \subset\,
{\mathcal P}{\mathcal M}^\alpha_s$ defined in \eqref{p2}. 
The codimension of its complement
$$
{\mathcal P}{\mathcal M}^\alpha_s\setminus M\, \subset\,
{\mathcal P}{\mathcal M}^\alpha_s
$$
is at least two.
\end{lemma}

\begin{proof}
We use the notation of \cite[p. 246, Proposition 1.2]{Bh}.
The codimension of ${\mathcal P}{\mathcal M}^\alpha_s\setminus M$
is bounded by that of $R^{\rm ss}\setminus R^{\rm s} \, \subset\,
R^{\rm ss}$. Hence the lemma follows from
\cite[p. 246, Proposition 1.2(3)]{Bh}.
\end{proof}

There is a universal projective bundle $\mathbb P$ over $X\times 
{\mathcal P}{\mathcal M}^\alpha_s$. It is universal in the sense
that each parabolic vector bundle $E_*\, \in\, {\mathcal P}{\mathcal 
M}^\alpha_s$, the restriction of $\mathbb P$ to $X\times\{E_*\}$
coincides with the projectivization ${\mathbb P}(E)$ of the
underlying vector bundle. To construct $\mathbb P$, we recall
that ${\mathcal P}{\mathcal M}^\alpha_s$ is constructed
as a geometric invariant theoretic quotient of a variety
$\mathcal R$ (see \cite{MY}, \cite{MS}). There is a canonical
universal projective bundle over $X\times \mathcal R$ admitting
a lift of the action of the group. The projective bundle
$\mathbb P$ is the corresponding quotient. (See also \cite[p. 6,
Theorem 2.7]{BBNN}.) The universal projective bundle $\mathbb P$
is unique.

Let ${\mathbb P}_{x_0}\, \longrightarrow\,
{\mathcal P}{\mathcal M}^\alpha_s$ be the projective bundle
obtained by restricting the above projective bundle to
$\{x_0\}\times {\mathcal P}{\mathcal M}^\alpha_s$, where
$x_0$ is the fixed point of $X$ (see \eqref{i1.}).
The restriction of ${\mathbb P}_{x_0}$ to the open subset
$M$ (see \eqref{p2}) is isomorphic to the projective
bundle ${\mathbb P}'$ defined in \eqref{is.}.

As mentioned in the introduction, any moduli space of stable
parabolic bundles is smooth (the obstruction to smoothness of
a stable point is $H^2$ of the sheaf of parabolic endomorphisms,
and this $H^2$ vanishes because the base is a curve).
Therefore, we conclude that in the special case under
consideration, Theorem \ref{thm0} follows from Lemma
\ref{lemma1} and Lemma \ref{lem2a}.

\section{One parabolic point and arbitrary weights}\label{step3}

We continue with the assumption that
there is exactly one parabolic point. But the earlier
condition on parabolic weights is now removed. As
before, we assume that $a_1\, \geq\, 2$.

As before, fix a rank, a determinant line bundle and
a quasi--parabolic filtration type. Let $\alpha$ 
and $\beta$ be two sets of parabolic weights for this
quasi--parabolic type with $\alpha$ being
sufficiently small. Let ${\mathcal P}{\mathcal M}^{\beta}$
(respectively, ${\mathcal P}{\mathcal M}^{\alpha}$) be the moduli
space of stable parabolic vector bundles with parabolic weights
$\beta$ (respectively, $\alpha$).

In Section \ref{smallweight} we have proved Theorem \ref{thm0}
for ${\mathcal P}{\mathcal M}^{\alpha}_s$. Therefore, from
Lemma \ref{thaddeus} we conclude that Theorem
\ref{thm0} holds also for ${\mathcal P}{\mathcal M}^{\beta}_s$.

\section{Multiple parabolic points}

In this section we drop the assumption in Section
\ref{step3} that there is only one parabolic point.

Fix $\ell$ parabolic points as in \eqref{ppoint}. At each
parabolic point $p_i$, fix the
quasi--parabolic structure of type
$\{r_{i,j}\}_{j=1}^{a_i}$ together with
the parabolic weights
$\alpha\, :=\, \{\alpha_{i,1}\, ,\cdots \, ,\alpha_{i,a_i}\}$
(see \eqref{n} and \eqref{fi.}). 

The proof of Theorem \ref{thm0} for multiple parabolic
points is similar to that for the case of one point.

First assume that all the parabolic weights are
sufficiently small. More precisely, assume that
all the parabolic weights are smaller than
$(n^2\ell)^{-1}$. This condition ensures that the underlying
vector bundle of a stable parabolic vector bundle is semistable.

We use the notation of Section \ref{smallweight}. Let
$\mathcal N \,\subset\, \overline{\mathcal N}$ be the moduli
space of stable vector bundles over $X$ of rank $n$ and
determinant $L$ (see \eqref{s}). Define
\begin{equation}\label{m1}
M\, := \,\pi_0^{-1}({\mathcal N})\, ,
\end{equation}
where $\pi_0\, : \, {\mathcal P}{\mathcal M}^{\alpha}_{s} 
\,\longrightarrow\, \overline{N}$ is the morphism that sends a
parabolic vector bundle to its underlying vector bundle.
Define
\begin{equation}\label{pi2}
\pi\, :=\, \pi_0\vert_M\, :\, M\, :=\,
\pi^{-1}_0({\mathcal N})\, \longrightarrow\, {\mathcal N}
\end{equation}
(see \eqref{p3}).

For each $i\, \in\, [1\, ,\ell]$, let
\begin{equation}\label{Pm}
P_i\, \subset\, \text{SL}(n, {\mathbb C})
\end{equation}
be the parabolic subgroup that preserves a fixed filtration
of subspaces
$$
{\mathbb C}^n \,= \, C_{1}^{i} \, \supsetneq\, \cdots\,
\, \supsetneq\, C_{a_i}^{i}
$$
such that $\dim C_{j}^{i}\, =\, \sum_{k=j}^{a_i} r_{i,k}$.

Since the parabolic weights are sufficiently small, any
quasi--parabolic structure on a stable vector bundle of
rank $n$ and determinant $L$ is parabolic stable.
Consequently, the projection $\pi$ in \eqref{pi2}
makes $M$ a fiber bundle over ${\mathcal N}$ with fiber
\begin{equation}\label{G}
{\mathbb F}\, :=\, \prod_{i=1}^\ell\text{SL}(n,
{\mathbb C})/P_i\, .
\end{equation}

Our first aim is to determine the Brauer group
$\text{Br}(M)$ of $M$ using the fibration $\pi$.

Consider the long exact sequence constructed as in
\eqref{eqn1} for the projection $\pi$ in \eqref{pi2}.
Since $\mathcal N$ is simply connected
\cite[p. 266, Proposition 1.2(b)]{BBGN},
the direct image
$$
R^1\pi_*{\mathbb G}_m \longrightarrow \mathcal N
$$ 
is canonically identified with the constant sheaf with stalk 
\begin{equation}\label{n0}
\text{Pic}({\mathbb F})\, =\, {\mathbb Z}^{N_0}\, ,
\end{equation}
where $N_0\, =\, \sum_{i=1}^\ell (a_i-1)$.

Since $\mathbb F$ in \eqref{G} is a smooth rational projective
variety, we have
\begin{equation}\label{lastterm}
H^2({\mathbb F}_{\et},\, {\mathbb G}_m) \, = \, 0\, . 
\end{equation}
In view of this and the above description of $R^1\pi_*{\mathbb G}_m$,
we may repeat the argument in the proof of Lemma \ref{lemma1}
for \eqref{e1} to conclude that
\begin{equation}\label{ac1}
(R^2\pi_*\mathbb G_m)_{\rm torsion}\,=\, 0\, .
\end{equation}
Now from \eqref{eqn1} we conclude that the Brauer class of
${\mathbb P}'$ generates $H^{2}(M_\et,\, {\mathbb G}_m)$,
where ${\mathbb P}'$ is constructed as in \eqref{is.}.

Therefore, from \eqref{eqn1} we have 
\begin{equation}\label{eqn22}
\longrightarrow\, \text{Pic}({\mathbb F})\, 
\stackrel{\theta}{\longrightarrow} \,
\text{Br}({\mathcal N}) \,\stackrel{\theta_1}{\longrightarrow}
\,\text{Br}(M) \,\longrightarrow\, 0\, .
\end{equation}

For each $i \in [1\, ,\ell]$, consider $P_i$ defined in
\eqref{Pm}. Let $\{\zeta_j^i\}_{2 \le j \le a_i}$ be the
generators of $\text{Pic}(\text{SL} (n,\mathbb C)/P_i)$ (see
\eqref{f2} for the description of these generators), and let
\begin{equation}\label{f33}
\eta :\, {\mathbb Z}^{\oplus N_0}\, 
\longrightarrow\,
\bigoplus_{i=1}^{\ell}\text{Pic}(\text{SL}(n, {\mathbb C})/P_i)
\, =\, \text{Pic}({\mathbb F})
\end{equation}
be the isomorphism defined by these generators (see \eqref{f3}
and \eqref{n0}).

For each $i \,\in \,[1\, ,\ell]$ and $j \,\in \, [2\, ,a_i]$, 
define 
$$
c_j^i\,:=\, \sum_{k=j}^{a_i}r_{i,k}\, .
$$

The homomorphism $\theta$ in \eqref{eqn22} satisfies the
following:
\begin{equation}\label{theta}
\theta(\zeta^i_j) \,= \, c^i_j \cdot\beta\, ,
\end{equation}
where $\beta\, \in\,\text{Br}({\mathcal N})$ is the Brauer
class of the Brauer--Severi variety ${\mathbb P}^0_{x_0}\,
\longrightarrow\, {\mathcal N}$ constructed as in \eqref{i1.}.
As before, define
\begin{equation}\label{theta2}
{\mathbb P}'\, :=\, \pi^*{\mathbb P}^0_{x_0}\, .
\end{equation}

The order of $\beta$ is $\text{g.c.d.}(n,d)$
\cite[p. 267, Theorem 1.8]{BBGN}.
Hence from \eqref{theta} we conclude that the order of
$\theta_1(\beta)$ is
\begin{equation}\label{mp}
m'\, :=\, \text{g.c.d.}(\text{g.c.d.}(n,d), \, c^1_2 \, ,
\cdots\, , c^1_{a_1}\, , \cdots \, , c^i_j \, , \cdots \, ,
c^\ell_{2}\, , \cdots \, , c^\ell_{a_\ell})\, ,
\end{equation}
where $\theta_1$ is the homomorphism in \eqref{eqn22}.

Consider $m$ defined in Theorem \ref{thm0}.
Using \eqref{n} it follows that $m$ coincides with
$m'$ defined in \eqref{mp}. Therefore, we have the following
proposition:

\begin{proposition}\label{prop2}
For the variety $M$ defined in \eqref{m1},
$$
{\rm Br}(M)\, =\, {\mathbb Z}/m{\mathbb Z}\, ,
$$
and ${\rm Br}(M)$ is generated by the Brauer class of the
Brauer--Severi variety ${\mathbb P}'\,\longrightarrow\, M$
constructed in \eqref{theta2}.
\end{proposition}

\begin{lemma}\label{lem2b}
Consider the Zariski open subset $M\, \subset\,
{\mathcal P}{\mathcal M}^\alpha_s$ defined in \eqref{m1}.
The codimension of the complement
$$
{\mathcal P}{\mathcal M}^\alpha_s\setminus M\, \subset\,
{\mathcal P}{\mathcal M}^\alpha_s
$$
is at least two.
\end{lemma}

\begin{proof}
The proof of Lemma \ref{lem2a} goes through without any change.
\end{proof}

{}From Proposition \ref{prop2} and Lemma \ref{lem2b}
we now conclude that Theorem
\ref{thm0} holds under the assumption that the
parabolic weights are sufficiently small.

For arbitrary parabolic weights, Theorem \ref{thm0}
is now deduced from the above special case using
Lemma \ref{thaddeus} (as done in Section \ref{step3}).
This completes the proof of Theorem \ref{thm0}.

\section{Existence of universal bundle}

As an application of Theorem \ref{thm0}, we will show
that there is a universal parabolic vector bundle over
$X\times {\mathcal P}{\mathcal M}^\alpha_s$ if and only if
$$
{\rm g.c.d.}(d\, , n\, ,r_{1,1}\, , \cdots\, ,
r_{1,a_1}\, , \cdots\, , r_{\ell,a_1}\, , \cdots\, ,
r_{\ell,a_\ell})\, =\, 1\, .
$$
This was proved earlier by N. Hoffmann \cite[Corollary 6.3]{Ho}.

If ${\rm g.c.d.}(d\, , n\, ,r_{1,1}\, , \cdots\, ,
r_{1,a_1}\, , \cdots\, , r_{\ell,a_1}\, , \cdots\, ,
r_{\ell,a_\ell})\, \not=\, 1$, then from Theorem \ref{thm0}
it follows immediately that there is no universal vector
bundle over $X\times {\mathcal P}{\mathcal M}^\alpha_s$;
in particular, there is no universal parabolic bundle.

To prove the converse, assume that
\begin{equation}\label{con.}
{\rm g.c.d.}(d\, , n\, ,r_{1,1}\, , \cdots\, ,
r_{1,a_1}\, , \cdots\, , r_{\ell,a_1}\, , \cdots\, ,
r_{\ell,a_\ell})\, =\, 1\, .
\end{equation}
Take any parabolic vector bundle
$$
E_*\, :=\, (E\, , \{F_{i,1} \supsetneq\cdots\,
\supsetneq F_{i,a_i}\}_{i=1}^\ell)\, \in\,
{\mathcal P}{\mathcal M}^\alpha_s\, .
$$
Consider the complex lines
$\bigwedge^{\text{top}}F_{i,j}$, where $i\, \in\, [1\, ,\ell]$
and $j\, \in\, [1\, ,a_i]$, together with the complex line
$$
\text{Det}\, E\, :=\, \bigwedge\nolimits^{\text{top}}
H^0(X,\, E)\bigotimes \bigwedge\nolimits^{\text{top}}
H^1(X,\, E)^*\, .
$$
Let $\{L_i\}_{i=1}^{m_0}$ be this collection of complex
lines.

Automorphisms of a stable parabolic vector bundle are
scalar multiplications. In other words,
$$
{\mathbb G}_m\, =\, \text{Aut}(E_*)\, .
$$
Note that $\text{Aut}(E_*)$
acts on the line $L_i$ for each $i\, \in\, [1\, ,m_0]$.

{}From \eqref{con.} it follows that there are integers
$e_i\, \in\, {\mathbb Z}$, $i\, \in\, [1\, ,m_0]$, with
some $e_i$ nonzero, such that
the group $\text{Aut}(E_*)$ acts trivially on the line
$$
\bigotimes_{i=1}^{m_0} L^{\otimes e_i}_i\, .
$$
Using this it can be shown that in the geometric invariant
theoretic construction of the moduli space ${\mathcal P}
{\mathcal M}^\alpha_s$, a suitable twist of the universal
parabolic vector bundle 
descends to $X\times {\mathcal P}{\mathcal M}^\alpha_s$.
(See \cite[p. 465, Proposition 3.2]{BY} for the details.)

\medskip
\noindent
\textbf{Acknowledgements.}\,
We are very grateful to the referee for comments to improve the
manuscript.

%%%%%%%%%%%%%%%%%%%%%%%%%%%%%%%%%%%%%%%%%%%%%%%%%%%%%%%%%%%%%%

\end{document}